\documentclass[12pt]{amsart}
\usepackage{amsmath,amsfonts,amssymb}
\usepackage[mathscr]{eucal}  % for script fonts
\usepackage{mathtools}% for xmapsto
\usepackage{hyperref}
\usepackage{verbatim}

\usepackage{bm} %for bold greek letters

\usepackage{caption}
\usepackage{subcaption}

\title{Einstein Tori and Crooked Surfaces}  % changed title 3/14/2015 pi day
\author[Burelle]{ Burelle, Jean-Philippe}
\author[Charette]{ Charette, Virginie}
\author[Francoeur]{Francoeur, Dominik}
\author[Goldman]{Goldman, William M.}
\address{{\it Burelle:\/} Universit\'e de Sherbrooke, Sherbrooke, Canada.}
\email[Burelle]{j-p.burelle@usherbrooke.ca}
\address{{\it Charette:\/} Universit\'e de Sherbrooke, Sherbrooke, Canada.}
\email[Charette]{virginie.charette@usherbrooke.ca}
\address{{\it Francoeur:\/} Universit\'e de Gen\`eve, Geneva, Switzerland.}
\email[Francoeur]{dominik.francoeur@unige.ch}
\address{{\it Goldman:\/} University of Maryland, College Park, MD 20742, USA.}
\email[Goldman]{wmg@math.umd.edu}
\thanks{Burelle, Charette and Francoeur gratefully acknowledge partial support from the Natural Sciences and Engineering Research Council of Canada.
We also gratefully acknowledge partial support from the US National
Science Foundation, in particular grants  DMS 1406281 and, especially
DMS 1107367 “Research Networks in the Mathematical Sciences: Geometric
structures And Representation varieties” (the GEAR Network).}
\date{\today}


\numberwithin{equation}{section}
\newtheorem{thm}{Theorem}
\newtheorem*{thm*}{Theorem}
\newtheorem{lem}[thm]{Lemma}
\newtheorem{cor}[thm]{Corollary}
\newtheorem{prop}{Proposition}

\newcommand{\Ein}[1]{\mathsf{Ein}^{#1}}
\newcommand{\R}{\mathbb{R}}

\newcommand{\ldot}[2]{#1 \cdot #2}
\newcommand{\Einth}{\Ein{3}}
\newcommand{\Einn}{\Ein{n}}

\newcommand{\Rnt}{\R^{n,2}}
\newcommand{\W}{\mathsf{W}}
\newcommand{\V}{\mathsf{V}}
\newcommand{\N}{\mathscr{N}}
\renewcommand{\O}{\mathsf{O}} % orthogonal group in

\renewcommand{\v}[1]{\mathbf{#1}}

\newcommand{\Proj}{\mathbb{P}}
\newcommand{\AdS}{\mathsf{AdS}}

\newcommand{\Det}{\mathsf{Det}}
\newcommand{\Adj}{\mathsf{Adj}}

\newcommand{\E}{\mathsf E}

\newcommand{\Ph}[1]{\mathsf{Pho}^#1}
\newcommand{\Li}{\mathscr L}
\newcommand{\Ht}{\mathsf{H}^2}

\newcommand{\graph}{\mathsf{graph}}

\newcommand{\Ker}{\mathsf{Ker}}
\newcommand{\Isom}{\mathsf{Isom}}
\newcommand{\Tr}{\mathsf{Tr}}
\newcommand{\Xx}{{\mathscr X}(x_0)}
\newcommand{\Hh}{\mathscr H}
\newcommand{\Yy}{\mathscr Y}
\newcommand{\SL}{\mathsf{SL}}
\newcommand{\SpfR}{\mathsf{Sp}(4,\R)}
\newcommand{\Sp}{\mathsf{Sp}}
\newcommand{\PSL}{\mathsf{PSL}}

\newcommand{\En}{\E^{n-1}}
\newcommand{\Enoo}{\E^{n-1,1}}

\newcommand{\Lag}{\mathsf{Lag}}

\begin{document}

\begin{abstract}
%Intersections of pairs of hyperplanes in the Einstein universe are classified,
%using a simple algebraic invariant generalizing the angle of intersection or distance
%between totally geodesic hyperplanes in hyperbolic space. In dimension $3$ this invariant is
%interpreted in terms of symplectic splittings of a $4$-dimensional real symplectic
%vector space. As an application, we give a disjointness criterion for crooked surfaces
%in the $3$-dimensional Einstein universe.
In hyperbolic space, the angle of intersection and distance classify pairs of totally geodesic hyperplanes. A similar algebraic invariant classifies pairs of hyperplanes in the Einstein universe. In dimension $3$, symplectic splittings of a $4$-dimensional real symplectic vector space model Einstein hyperplanes and the invariant is a determinant. The classification contributes to a complete disjointness criterion for crooked surfaces in the $3$-dimensional Einstein universe.
\end{abstract}
\maketitle
%\tableofcontents

\section{Introduction}
Polyhedra bounded by {\em crooked surfaces} form fundamental domains in the Einstein Universe for Lorentzian Kleinian groups (\cite{MR3231611}, \cite{MR2182704}).  Crooked surfaces are assembled from pieces of certain hypersurfaces, namely light cones and Einstein tori.  This motivates our study of  these hypersurfaces, and how they intersect.
%Describing polyhedra in the Einstein universe motivates the study of intersections of hypersurfaces which correspond to their faces. Light cones and Einstein tori are such hypersurfaces and they assemble into crooked surfaces. Polyhedra bounded by crooked surfaces form fundamental domains for Lorentzian Kleinian groups (\cite{MR3231611}, \cite{MR2182704}).

The theory of \emph{crooked planes}, in the context of Minkowski space, has been very successful in understanding and classifying discrete groups of affine transformations acting properly on $\mathbb{R}^3$ (\cite{MR3262435},\cite{MR2003687},\cite{MR3569564} and \cite{MR3480555}). Crooked planes are piecewise linear surfaces in Minkowski $3$-space which bound fundamental domains for proper affine actions. In 2003, Frances~\cite{MR1989275} studied the boundary at infinity of these quotients of Minkowski space by introducing the \emph{conformal compactification} of a crooked plane. In this paper, we call conformally compactified crooked planes \emph{crooked surfaces}.

Recently, Danciger-Gu\'eritaud-Kassel~\cite{DGKfundamentald} have adapted crooked planes to the negatively curved anti-de Sitter space. In a note shortly following the DGK paper, Goldman~\cite{G14} unified crooked planes and anti-de Sitter crooked planes. More precisely, Minkowski space and anti-de Sitter space can be conformally embedded in the Einstein universe in such a way that crooked planes in both contexts are subsets of a crooked surface.

A crooked surface is constructed using three pieces : two wings, and a stem. The wings are parts of light cones, and the stem is part of an Einstein torus. In order to understand the intersection of crooked surfaces, we first focus on Einstein tori. Our first result classifies their intersections.
\begin{thm}
\label{ToriIntersection}
Let $T_1,T_2\subset \Einth$ be Einstein tori.
Suppose that $T_1\neq T_2$.
Then $T_1\cap T_2$ is nonempty, and exactly one of the following possibilities
occurs:
\begin{itemize}
\item $T_1\cap T_2$ is a union of two photons which intersect in exactly one point.
\item $T_1\cap T_2$ is a spacelike circle and the intersection is transverse.
\item $T_1\cap T_2$ is a timelike circle and the intersection is transverse.
\end{itemize}
\noindent
A single geometric invariant $\eta(T_1,T_2)$, related to the Maslov index, distinguishes the three cases.
\begin{itemize}
  \item $\eta(T_1,T_2)=1$ if and only if  $T_1\cap T_2$ is a union of two photons which intersect in exactly one point,
  \item $\eta(T_1,T_2) > 1$ if and only if $T_1\cap T_2$ is spacelike, and
  \item $\eta(T_1,T_2) < 1$ if and only if $T_1\cap T_2$ is timelike.
\end{itemize}
\end{thm}
We next show how to further interpret this result in the three-dimensional case. The Lagrangian Grassmannian in dimension $4$ is a model of the $3$-dimensional Einstein universe. The relationship between the two models was studied extensively in \cite{MR2436232}. We develop the theory of Einstein tori in the space of Lagrangians and characterize $\eta$ as the determinant of a linear map.

A simple consequence of Theorem \ref{ToriIntersection} is
\begin{cor}
Let $T_1,T_2$ be a pair
 Einstein tori. Then, $T_1\cap T_2$ is non-contractible as a subset of $T_1$ or $T_2$.
\end{cor}

We use this corollary to prove a complete disjointness criterion for crooked surfaces,
generalizing the construction  in Charette-Francoeur-Lareau-Dussault~\cite{MR3231611}
and the criterion for disjointness of anti-de Sitter crooked planes in Danciger-Gu\'eritaud-Kassel~\cite{DGKfundamentald} :

\begin{thm}
  \label{surfacecriterion}
  Two crooked surfaces $C,C'$ are disjoint if and only if the four photons on the boundary of the stem of $C$ are disjoint from $C'$, and the four photons on the boundary of the stem of $C'$ are disjoint from $C$.
\end{thm}
The Lagrangian model of the Einstein universe allows us to express the condition in this theorem explicitly in terms of symplectic products. In that model, a pair of simple inequalities guarantee that a photon does not intersect a crooked surface.

Finally, we show that the criterion in theorem \ref{surfacecriterion} reduces to the criterion for disjointness of anti-de Sitter crooked planes from \cite{DGKfundamentald} when specializing to crooked surfaces adapted to an anti-de Sitter patch.

\section*{Notations and terminology}

If $V$ is a vector space, denote the associated
{\em projective space\/} $\Proj(V)$, defined as the space
of all $1$-dimensional linear subspaces of $V$.
If $\v{v}\in V$ is a nonzero vector in a vector space $V$,
then denote the corresponding point (projective equivalence class) in the projective space $\Proj(V)$ by $[\v{v}]\in\Proj(V)$. We call a real vector space endowed with a nondegenerate bilinear form a \emph{bilinear form space}.
If $\v{v}\in V$ is a nonzero vector in a bilinear form space $(V,\ldot{}{})$,
then
\[ \v{v}^\perp\ := \ \{ \v{w}\in V \mid \ldot{\v{v}}{\v{w}} = 0 \} \]
is a linear hyperplane in $V$.
When $\v{v}$ is non-null, then $\v{v}^\perp$ is nondegenerate and
defines an orthogonal decomposition
\[ V = \R \v{v} \oplus \v{v}^\perp. \]
More generally, if $S\subset V$ is a subset, then define
\[ S^\perp\ := \ \{ \v{w}\in V \mid \ldot{\v{v}}{\v{w}} = 0, \quad \forall \v{v}\in S\}.\]

%Finally, in this section we record some facts about the main invariant characterizing
%pairs of Einstein tori.
%These involve the function
%\[
%d(x)  := (1+x)/(1-x).
%\]
%Then
%%Let $d = (1+x)/(1-x)$. Then
%\begin{align*}
%d(x) = \infty &\Longleftrightarrow  x = 1 \\
%d(x) < 0 &\Longleftrightarrow  x > 1 \quad\text{or}\quad x < -1 \\
%d(x) = 0 &\Longleftrightarrow  x = -1 \\
%d(x) > 0 &\Longleftrightarrow 1 >  x > -1 \\
%d(x) = \infty &\Longleftrightarrow x = 1\end{align*}
%Numerical properties of $d(x)$ relate to qualitative properties of pairs
%of Einstein tori, and we found that it is useful to recall elementary properties
%of this function.

\section{Einstein geometry}
This section briefly summarizes the basics of the geometry of $\Einn$.
For more details, see \cite{MR2182704,MR2436232,MR1989275,MR3231611}.

\subsection{The bilinear form space $\R^{n,2}$}

Let $\W$ be a $(n+2)$-dimensional real vector space endowed with a signature
$(n,2)$ symmetric bilinear form
\begin{align*}
\W \times \W & \longrightarrow \R \\
(\v{u},\v{v}) &\mapsto \ldot{\v{u}}{\v{v}}.
\end{align*}
Define the {\em null cone\/}:
\[\N(\W) := \{ \v{v}\in \W \mid \ldot{\v{v}}{\v{v}} = 0\}.\]
The \emph{Einstein universe} is the projectivization of $\N(\W)$ :
\[\Einn:=\Proj\big(\N(\W)\big).\]
$\Einn$ carries a natural conformal Lorentzian structure coming from the product on $\W$. More precisely,
smooth cross-sections of  the quotient map $\N(\W)\longrightarrow\Einn$ determine Lorentzian structures on $\Einn$.
Furthermore these Lorentzian structures are  conformally equivalent
to each other.

The orthogonal group $\O(n,2)$ of $\W$ acts conformally and transitively on
$\Einn$. In fact, the group of conformal automorphisms of $\Einn$ is exactly $\O(n,2)$.

\subsection{Photons and light cones}
A {\em photon\/} is the projectivization $\Proj(P)$ of a totally isotropic $2$-plane
$P\subset\W$. It corresponds to a lightlike geodesic in the conformal Lorentzian metric of $\Einn$.
A {\em spacelike circle\/} (respectively {\em timelike circle\/})
is the projectived null cone
$\Proj\left(\N(S)\right)$ of a subspace $S\subset\W$ which has signature
$(2,1)$ (respectively signature $(1,2)$).

A {\em light cone\/} is the projectivized null cone $\Proj \left(\N(H)\right)$
of a degenerate hyperplane $H\subset\W$.
Such a degenerate hyperplane $H = \v{n}^\perp$ for some null vector
$\v{n}\in\N(\W)$. In terms of the synthetic geometry of $\Einn$,
the light cone defined by $p = [\v{n}]\in\Einn$ equals the union of all
photons containing $p$. We will denote it by $\Li(p)$.

One can consider a different homogeneous space, the \emph{space of photons} of $\Einn$, denoted $\Ph{n}$. It admits a natural contact structure (see \cite{MR2436232}) in which the photons in a lightcone form a Legendrian submanifold. The contact geometry of photon space is intimately related to the conformal Lorentzian geometry of the Einstein universe. This relation stems from the incidence relation between the two spaces. We say that a point $p\in \Einn$ is \emph{incident} to a photon $\phi\in\Ph{n}$ whenever $p\in\phi$. By extension, two points $p,q\in\Einn$ are called incident when they are incident to a common photon, and two photons $\phi,\psi\in\Ph{n}$ are called incident when they intersect in a common point.

\subsection{Minkowski patches}
\label{minkowski_patches}
The complement in $\Einn$ of a light cone is a {\em Minkowski patch.\/}
Its natural structure is {\em Minkowski space $\Enoo$,\/}
an affine space with a parallel Lorentzian metric.
Any geodesically complete simply-connected flat Lorentzian manifold
is isometric to $\Enoo$.
As such it is the model space for flat Lorentzian geometry.

Following \cite{MR2436232}, for $\W$ we use quadratic form of signature $(n,2)$ given by
\[\ldot{\v{v}}{\v{v}} := v_1^2 + v_2^2 + \dots + v_{n-1}^2 - v_n^2 - v_{n+1}v_{n+2}\] and work in the embedding of Minkowski space
\begin{align}\label{eq:Embedding}
\Enoo &\longrightarrow \Ein{n} \notag \\
\bmatrix \v{v} \\ v^n \endbmatrix &\longmapsto
\bmatrix \v{v} \\ v^n \\ \Vert \v{v}\Vert^2 - (v^n)^2  \\ 1\endbmatrix.
\end{align}
In the expression above,
\[ \v{v} := \bmatrix v^1 \\ \vdots \\ v^{n-1} \endbmatrix \ \in\ \En \]
is a vector in Euclidean space with Euclidean norm
$\Vert \v{v}\Vert$, and the Lorentzian norm in $\Enoo$  is:
\[ (\v{v}, v^n) \longmapsto  \Vert\v{v}\Vert^2 - (v^n)^2. \]

The complement of this embedding of $\Enoo$ is a light cone, and we will denote its vertex by $p_\infty$. This vertex is called the \emph{improper point} in \cite{MR2436232}, and its coordinates in a basis as above are:
\[ p_\infty \longleftrightarrow \bmatrix \v{0} \\ 0 \\ 1 \\ 0 \endbmatrix.\]
The closure in $\Einn$ of every non-null geodesic $\gamma$ in $\Enoo$ contains $p_\infty$
and the union $\gamma \cup \{p_\infty\}$ is a spacelike circle
or a timelike circle according to the nature of $\gamma$.
Conversely every timelike or spacelike circle which
contains $p_\infty$ is the closure of a timelike or spacelike geodesic in $\Enoo$.

The light cone of a point which is not $p_\infty$, but belongs to its light cone, intersects the
Minkowski patch $\Enoo$ in an affine hyperplane
upon which the Lorentzian structure on $\Enoo$ restricts to a field of degenerate
quadratic forms,  that is,  a {\em null hyperplane.}
% The intersection of this light cone with the light cone at infinity
% is a photon which contains the point $p_\infty$.

If we choose an origin $p_0$ for a Minkowski patch, then we get an identification
of the patch with a Lorentzian vector space. The trichotomy of vectors into timelike,
spacelike and lightlike has an intrinsic interpretation with respect to $p_0$ and $p_\infty$:

A point is :
\begin{enumerate}
\item timelike if it lies on some timelike circle through $p_0$ and $p_\infty$,
\item spacelike if it lies on some spacelike circle through $p_0$ and $p_\infty$, and
\item lightlike if it lies on a photon through $p_0$.
\end{enumerate}

One and only one of these three happens for every point in the Minkowski patch.

\subsection{Einstein hyperplanes}
\label{sec:EinsteinHyperplanes}
An {\em Einstein hyperplane\/}  $H$ corresponds to a linear hyperplane $ \ell^\perp \subset\Rnt$ orthogonal to a spacelike line $\ell\subset\Rnt$.
A linear hyperplane $\ell^\perp$ is conveniently described by a
{\em normal vector $\v{s}\subset\ell$, \/} which we may assume satisfies
$\ldot{\v{s}}{\v{s}} = 1$. In that case $\v{s}$ is determined up to multiplication by
$\pm 1$.

The hyperplane $\v{s}^\perp$ is a bilinear form space isomorphic to
$\R^{n-1,2}$ and its projectivized null cone is a model for $\Ein{n-1}$.
In dimension $n=3$, an Einstein hyperplane is homeomorphic to a $2$-torus
$S^1\times S^1$ so we will call it an {\em Einstein torus}.
Under the embedding \eqref{eq:Embedding},
an Einstein hyperplane which passes through the point $p_\infty$
meets the Minkowski patch $\Enoo$ in an affine hyperplane upon which the Lorentzian structure on $\Enoo$ restricts to a Lorentzian metric, that is, a {\em timelike\/} hyperplane.

Since an Einstein torus is a totally geodesic embedded copy of $\Ein{2}$, it has a pair of natural foliations by photons. This is because the light cone of a point in $\Ein{2}$ is a pair
of photons through that point.
As described in Goldman \cite{G14} for $n=3$,
the complement of an Einstein hyperplane has the natural structure of the
double covering of {\em anti-de Sitter space.} This identification is presented in more detail
in section \ref{sec:AdS}.

\section{Pairs of Einstein hyperplanes}
The purpose of this section is to define the invariant $\eta\ge 0$ characterizing
pairs of hyperplanes in $\Einn$ and to prove theorem \ref{ToriIntersection}. We describe the moduli space of equivalence
classes of pairs, and reduce to the case $n=3$.
Then \S\ref{sec:Symplectic} reinterprets $\Einth$ in terms of symplectic geometry
using the local isomorphism $\SpfR \longrightarrow\O(3,2)$.

\subsection {Pairs of positive vectors}
A linearly independent pair of two unit-spacelike vectors $\v{s}_1, \v{s}_2$
spans a $2$-plane $\langle \v{s}_1,\v{s}_2\rangle \subset \W$ which is:
\begin{itemize}
\item Positive definite $\Longleftrightarrow$ $\vert\ldot{\v{s}_1}{\v{s}_2} \vert < 1$;
\item Degenerate $\Longleftrightarrow$ $\vert\ldot{\v{s}_1}{\v{s}_2} \vert = 1$;
\item Indefinite $\Longleftrightarrow$ $\vert\ldot{\v{s}_1}{\v{s}_2} \vert > 1$.
\end{itemize}
The positive definite and indefinite cases respectively determine orthogonal splittings :
\begin{align*}
\W  & \cong \R^{n,2} = \R^{2,0} \oplus \R^{n-2,2} \\
\W  & \cong \R^{n,2} = \R^{1,1} \oplus \R^{n-1,1}.
\end{align*}
In the degenerate case, the null space is spanned by
$\v{s}_1 \pm \v{s}_2$, where
\[ \ldot{\v{s}_1}{\v{s}_2} = \mp 1.\]
By replacing $\v{s}_2$ by $-\v{s}_2$ if necessary, we may assume
that $\ldot{\v{s}_1}{\v{s}_2} = 1$.
Then $\v{s}_1 - \v{s}_2$ is null.
Since $\W$ itself is nondegenerate, there exists $\v{v}_3\in\W$ such
that
\[ \ldot{(\v{s}_1 -\v{s}_2)}{\v{v}_3} = 1.\] Then  $\v{s}_1,\v{s}_2,\v{v}_3$ span
a nondegenerate $3$-plane of signature $(2,1)$.

In all three cases, there is a $5-$dimensional subspace of signature $(3,2)$ containing $\v{s}_1$ and $\v{s}_2$.
For that reason, the discussion of pairs of Einstein hyperplanes can be
reduced to the case of $n=3$.

The absolute value of the product
\[ \eta(H_1, H_2) := \vert \v{s}_1\cdot \v{s}_2 \vert \]
is a nonnegative real number, depending only on the pair
of Einstein hyperplanes $H_1$ and $H_2$.
Specifying the above discussion to the case $n=3$ we have proved Theorem \ref{ToriIntersection} :
\begin{itemize}
\item If the span of $\v{s}_1,\v{s}_2$ is positive definite ($\eta(H_1,H_2)<1$), then the intersection of the corresponding Einstein
tori is the projectivised null cone of a signature $(1,2)$ subspace, which is a timelike circle.
\item If the span of $\v{s}_1,\v{s}_2$ is indefinite ($\eta(H_1,H_2)>1$), then the intersection is the projectivised null cone
of a signature $(2,1)$ subspace, which is a spacelike circle.
\item Finally, if the span of $\v{s}_1,\v{s}_2$ is degenerate ($\eta(H_1,H_2)=1$), the span $\R \v{s}_1 + \R \v{s}_2$ is a degenerate $2$-plane in $\R \v{s}_1 + \R \v{s}_2 + \R \v{v}_3 \cong \R^{2,1}$. The orthogonal complement of this $\R^{2,1}$ is of signature $(1,1)$ and \emph{is contained} in $(\R \v{s}_1 + \R \v{s}_2)^\perp = \v{s}_1^\perp \cap \v{s}_2^\perp$. Since this last subspace is of dimension $3$ and must also contain the degenerate direction of $\R \v{s}_1 + \R \v{s}_2$, it is a degenerate subspace with signature $(+,-,0)$. Its null cone is exactly the union of two isotropic planes intersecting in the degenerate direction, so when projectivising we get a pair of photons intersecting in a point.
\end{itemize}
\begin{cor}\label{Cor:IntersectionNonContractible}
The intersection of two Einstein tori is non-contractible in each of the two tori.
\begin{proof}
An Einstein torus is a copy of the $2$-dimensional Einstein universe. Explicitly, we can write it as $\mathbb{P}(\N)$ where $\N$ is the null cone in $\mathbb{R}^{2,2}$. A computation shows that all timelike circles are homotopic, all spacelike circles are homotopic and these two homotopy classes together generate the fundamental group of the torus. Similarly, photons are homotopic to the sum of these generators.
\end{proof}
\end{cor}
\begin{figure}
    \centering
    \begin{subfigure}[b]{0.3\textwidth}
        \includegraphics[width=\textwidth]{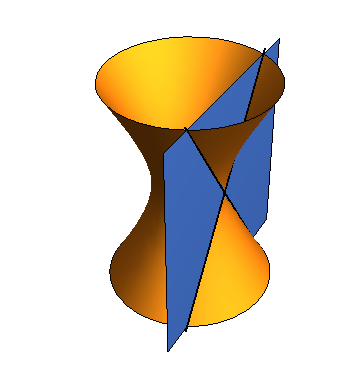}
        \caption{Two photons}
        \label{fig:photonsl}
    \end{subfigure}
    ~
    \begin{subfigure}[b]{0.3\textwidth}
        \includegraphics[width=\textwidth]{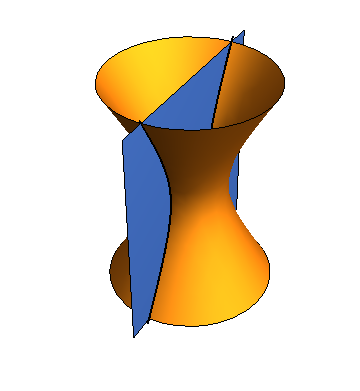}
        \caption{A timelike circle}
        \label{fig:timelike}
    \end{subfigure}
    ~
    \begin{subfigure}[b]{0.3\textwidth}
        \includegraphics[width=\textwidth]{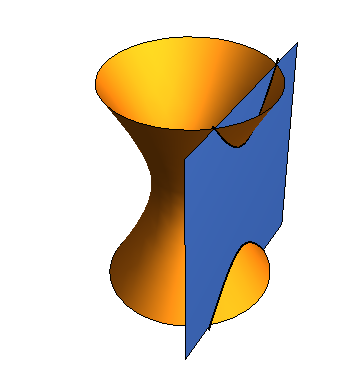}
        \caption{A spacelike circle}
        \label{fig:spacelike}
    \end{subfigure}
    \caption{The three possible types of intersection for a pair of Einstein tori, viewed in a Minkowski patch.}\label{fig:intersections}
\end{figure}
\subsection{Involutions in Einstein tori}
Orthogonal reflection in $\v{s}$ defines an involution of $\Einn$ which fixes the corresponding
hyperplane $H=\v{s}^\perp$.

The orthogonal reflection in a positive vector $\v{s}$ is defined by:
\[R_{\v{s}}(\v{v}) = \v{v} - 2\frac{\ldot{\v{v}}{\v{s}}}{\ldot{\v{s}}{\v{s}}}\v{s}.\]
We compute the eigenvalues of the composition $R_{\v{s}}R_{\v{s'}}$, where $\v{s},\v{s'}$ are unit spacelike vectors, and relate this to the invariant $\eta$.

The orthogonal subspace to the plane spanned by $\v{s}$ and $\v{s'}$ is fixed pointwise by this composition. Therefore, $1$ is an eigenvalue of multiplicity $n$. In order to determine the remaining eigenvalues, we compute the restriction of $R_{\v{s}}R_{\v{s'}}$ to the subspace $\R \v{s} + \R \v{s'}$.
\begin{align*}
R_{\v{s}}R_{\v{s'}}(\v{s}) &= R_{\v{s}}(\v{s} - 2(\ldot{\v{s}}{\v{s'}})\v{s'})\\
&=-\v{s} - 2(\ldot{\v{s}}{\v{s'}})(\v{s'}-2(\ldot{\v{s'}}{\v{s}})\v{s})\\
&=(4 (\ldot{\v{s'}}{\v{s}})^2-1)\v{s} -2(\ldot{\v{s'}}{\v{s}})\v{s'}.\\
\\
R_{\v{s}}R_{\v{s'}}(\v{s'}) &= R_{\v{s}}(-\v{s'})\\
&= -\v{s'} + 2(\ldot{\v{s}}{\v{s'}})\v{s}.
\end{align*}
The matrix representation of $R_{\v{s}}R_{\v{s'}}$ in the basis $\v{s},\v{s'}$ is therefore:
\[\begin{pmatrix}
4 (\ldot{\v{s'}}{\v{s}})^2-1 & 2(\ldot{\v{s}}{\v{s'}})\\
-2(\ldot{\v{s'}}{\v{s}}) & -1
\end{pmatrix}.
\]
The eigenvalues of this matrix are:
\[2(\ldot{\v{s}}{\v{s'}})^2 - 1 \pm 2 (\ldot{\v{s}}{\v{s'}})\sqrt{(\ldot{\v{s}}{\v{s'}})^2-1}.\]
We observe that they only depend on the invariant $\eta=|\ldot{\v{s}}{\v{s'}}|$. The composition of involutions has real distinct eigenvalues when the
intersection is spacelike, complex eigenvalues when the intersection is timelike, and a double real eigenvalue when the
intersection is a pair of photons.

The case when $\v{s}_1\cdot\v{s}_2 = 0$ is special:
in that case the two involutions commute and we will say that the Einstein hyperplanes are \emph{orthogonal}. As observed at the end of section \ref{sec:EinsteinHyperplanes}, the complement of an Einstein torus in $\Einth$ is
a model for the double covering space of
{\em anti-de Sitter space\/} $\AdS^3$, which has a complete Lorentzian metric of constant curvature $-1$.
In this conformal model of $\AdS^3$ (see \cite{G14}),
indefinite totally geodesic $2$-planes are represented by tori which are orthogonal to $\partial\AdS^3$.

\section{The Symplectic model}\label{sec:Symplectic}
We describe a model for Einstein $3$-space in terms of $4$-dimensional symplectic algebra,
an alternative approach which is simpler for some calculations.

Let $(\V,\omega)$ be a $4$-dimensional {\em real symplectic vector space,\/}
that is, $\V$ is a real vector space of dimension $4$ and
$\V\times \V\xrightarrow{~\omega~} \R$ is a nondegenerate skew-symmetric
bilinear form. Let $\mathsf{vol}\in\Lambda^4(V)$ be the element defined by the equation $(\omega\wedge\omega)(\mathsf{vol})=-2$. The second exterior power $\Lambda^2(\V)$ admits a nondegenerate symmetric bilinear form $\cdot$
of signature $(3,3)$ defined by
\[(\v{u}\wedge\v{v})\wedge(\v{u}'\wedge\v{v}')=(\v{u}\wedge\v{v})\cdot(\v{u}'\wedge\v{v}')\mathsf{vol}.\]
The kernel
\[ \W := \Ker(\omega)\subset \Lambda^2(\V)\]
inherits a symmetric bilinear form which has signature $(3,2)$.

Define the vector $\omega^*\in \Lambda^2\V$ to be \emph{dual} to $\omega$ by the equation
\[\omega^*\cdot(\v{u}\wedge\v{v}) = \omega(\v{u},\v{v}),\]
for all $\v{u},\v{v}\in \V$.
Because of our previous choice of $\mathsf{vol}$, we have
$\ldot{\omega^*}{\omega^*}=-2$.
The bilinear form $\ldot{}{}$, together with the vector $\omega^*$ define a \emph{reflection}
\begin{align*}
  \mathsf{R}_{\omega^*}:\Lambda^2(\V)&\rightarrow\Lambda^2(\V)\\
  \alpha &\mapsto \alpha + (\ldot{\alpha}{\omega^*})\omega^*.
\end{align*}
The fixed set of this reflection is exactly the vector subspace $\W$ orthogonal to $\omega^*$.

The \emph{Pl\"ucker embedding} $\iota : \mathsf{Gr}(2,\V) \rightarrow \Proj(\Lambda^2(\V))$ maps $2$-planes in $\V$ to lines in $\Lambda^2(\V)$. We say that a plane in $\V$ is \emph{Lagrangian} if the form $\omega$ vanishes identically on pairs of vectors in that plane. If we restrict $\iota$ to Lagrangian planes, then the image is exactly the set of null lines in $\W$.

The form $\omega$ yields a relation of (symplectic) orthogonality on $2$-planes in $\V$. Lagrangian planes are orthogonal to themselves, and non-Lagrangian planes have a unique orthogonal complement which is also non-Lagrangian. The following proposition relates orthogonality in $\V$ with an operation on $\Lambda^2(\V)$.
\begin{prop}
  \label{orthogonalrefl}
  A pair of $2$-dimensional subspaces $S,T\subset \V$ are orthogonal with respect to $\omega$ if and only if $[\mathsf{R}_{\omega^*}(\iota(S))] = [\iota(T)]$.
  \begin{proof}
    First, assume $S$ is Lagrangian. This means that $S=S^\perp$, and that $\iota(S)\in \omega^{*\perp}$. Hence,
    \[\mathsf{R}_{\omega^*}(\iota(S)) = \iota(S)=\iota(S^\perp).\]

    Next, if $S$ is not Lagrangian, then we can find bases $(\v{u},\v{v})$ of $S$ and $(\v{u}',\v{v}')$ of $S^\perp$ satisfying $\omega(\v{u},\v{v})=\omega(\v{u}',\v{v}')=1$ and all other products between these four are zero. Then,
    \[\mathsf{vol}=-\v{u}\wedge \v{v}\wedge \v{u}'\wedge \v{v}'\]
    and
    \[\omega^*=-\v{u}\wedge \v{v} -\v{u}'\wedge \v{v}'.\]
    Consequently,
    \[[\mathsf{R}_{\omega^*}(\iota(S))] = [\v{u}\wedge\v{v} + \omega(\v{u},\v{v})\omega^*] = [-\v{u}'\wedge\v{v}'] = [\iota(S^\perp)].\qedhere\]
  \end{proof}
\end{prop}

\subsection{Symplectic interpretation of Einstein space and photon space}
%{\em Einstein space\/} $\Einth$ is defined as the projectivized null cone in $\W$, consisting of null lines in $\W$.
%It is topologically the quotient of $S^2 \times S^1$ by the diagonal
%antipodal map, and is homeomorphic to the mapping torus
%of the antipodal map on $S^2$.
%{\em Photon space\/} $\Ph$ is the space of totally isotropic
%$2$-planes in $\W$. Alternatively, $\Einth$ consists of Lagrangian $2$-planes in $\V$, and $\Ph$ is the projective space
%$\mathbb{P}(\V)$,  consisting of lines in $\V$ (for details see \cite{MR2436232}).

The natural incidence relation between $\Einth$ and $\Ph{3}$ is
described in the two algebraic models ($\V$ and $\W$) as follows.
A point $p\in\Einth$ and a photon $\phi\in\Ph{3}$ are
{\em incident\/} if and only if $(p,\phi)$ satisfies one
of the two equivalent conditions:
\begin{itemize}
\item The null line in $\W$ corresponding to $p$ lies in
the isotropic $2$-plane in $\W$ corresponding to $\phi$.
\item The Lagrangian $2$-plane in $\V$ corresponding to $p$ contains the line in $\V$ corresponding to $\phi$.
\end{itemize}
\noindent
These two are equivalent because of the following proposition :
\begin{prop}
  \label{transversePQ}
  Let $P,Q\subset \V$ be two-dimensional subspaces. Then, $P\cap Q= 0$ if and only if $\ldot{\iota(P)}{\iota(Q)}\neq0$.
  \begin{proof}
    Choose bases $\v{u},\v{v}$ of $P$ and $\v{u}',\v{v}'$ of $Q$. Then,
    \[\v{u}\wedge\v{v}\wedge\v{u}'\wedge\v{v}' \neq 0\]
    if and only if $\v{u},\v{v},\v{u}',\v{v}'$ span $\V$ which is equivalent to $P$ and $Q$ being transverse.
  \end{proof}
\end{prop}
The {\em light cone\/} $\Li(p)$ of a point $p\in\Einth$ is the union of
all photons containing $p$.
It corresponds to the orthogonal hyperplane $[p]^\perp\subset \W$ of the null line corresponding to $p$.
In photon space $\Proj(\V)$, the photons containing
$p$ form the projective space $\Proj(L)$ of the Lagrangian
$2$-plane $L$ corresponding to $p$.
\subsection{Timelike or spacelike triples and the Maslov index}
Fixing a pair of non-incident points in the Einstein universe induces a trichotomy on points, as explained in section \ref{minkowski_patches}.
The corresponding data in the Lagrangian model is related to the \emph{Maslov index} of a triple of Lagrangians.

Two non-incident points correspond to a pair of transverse Lagrangians $L,L'$. This induces a splitting $\V=L\oplus L'$.
Together with the symplectic form $\omega$, this splitting defines a quadratic form defined by
\[q_{L,L'}(\v{v}) := \omega(\pi_L(\v{v}),\pi_{L'}(\v{v})).\]

The \emph{Maslov index} of a triple of pairwise transverse Lagrangians $L,P,L'$ is the integer $m(L,P,L')=\mathrm{sign}(q_{L,L'}|_P)$, where $\mathrm{sign}(q)$  is the difference between the number of positive and negative eigenvalues of $q$. Transversality
implies that $q_{L,L'}$ restricted to $P$ is nondegenerate. This index classifies orbits of triples of pairwise transverse Lagrangians \cite{MR2264460}.

Lagrangians which are nontransverse to $L$ correspond to lightlike points, Lagrangians $P$ with $|m(L,P,L')|=2$
correspond to timelike points, and Lagrangians $P$ with $m(L,P,L')=0$ correspond to spacelike points.
\subsection{Nondegenerate planes and symplectic splittings}
\label{sec:ssplit}
We describe the algebraic structures equivalent to an
{\em Einstein torus\/} in $\Einth$. As a reminder, these are
hyperplanes of signature $(2,2)$ inside $\W \cong \R^{3,2}$,
and describe surfaces in $\Ein{3}$ homeomorphic to a $2$-torus.

In symplectic terms, an Einstein torus corresponds to a splitting
of $\V$ as a symplectic direct sum of two nondegenerate
$2$-planes. Let us detail this correspondence.

Define a $2$-dimensional subspace $S\subset\V$ to be {\em nondegenerate\/} if and only if
the restriction $\omega|_S$ is nondegenerate.
A nondegenerate $2$-plane $S\subset \V$ determines a splitting as follows. The plane
\[ S^\perp := \{ \v{v}\in\V \mid \omega(\v{v},S) = 0\} \]
is also nondegenerate, and defines a {\em symplectic complement\/} to $S$.
In other words, $\V$ splits as an (internal) symplectic direct sum:
\[ \V = S \oplus S^\perp.\]
The corresponding Einstein torus is then the set of Lagrangians which are non-transverse
to $S$ (and therefore also to $S^\perp$).

The lines in $S$ determine a projective line in $\Ph{3}$ which is
{\em not\/} Legendrian. Conversely, non-Legendrian
projective lines in $\Ph{3}$ correspond to nondegenerate $2$-planes. This non-Legendrian line in $\Ph{3}$, as a set of photons, corresponds to one of the two rulings of the Einstein torus.
The other ruling corresponds to the line $\Proj(S^\perp)$.

In order to make explicit the relationship between the descriptions of Einstein tori in the two models, define a map $\mu$ as follows:
\begin{align*}\mu : \mathsf{Gr}(2,\V) &\rightarrow \Proj(\W)\\
                       S &\mapsto \left[\widehat{\iota(S)} + \frac12 \omega(\widehat{\iota(S)})\omega^*\right],
\end{align*}
where $\widehat{\iota(S)}$ is any representative in $\W$ of the projective class $\iota(S)$.

The map $\mu$ is the composition of the Pl\"ucker embedding $\iota$ with the orthogonal projection onto $\W$.
\begin{lem}
  \label{symplecticspacelike}
  For $S$ a nondegenerate plane, the image of $\mu$ is always a spacelike line, and $\mu(S)=\mu(S^\perp)$.
  \begin{proof}
    For the first part, let $\v{s}$ be any vector representative of the line $\iota(S)$. Then,
    \[\left(\v{s} + \frac12\omega(\v{s})\omega^*\right) \wedge \left(\v{s} +  \frac12\omega(\v{s})\omega^*\right) = \frac12\omega(\v{s})^2 \mathsf{vol},\]
    and therefore $\mu(S)$ is spacelike.

    The second part is a consequence of the correspondence between orthogonal complements and reflection in $\omega^*$ (Proposition \ref{orthogonalrefl}) and the fact that a vector and its reflected copy have the same orthogonal projection to the hyperplane of reflection.
  \end{proof}
\end{lem}
\begin{prop}
  The map $\mu$ induces a bijection between spacelike lines in $\W$ and symplectic splittings of $\V$. Under the Pl\"ucker embedding $\iota$, the Einstein torus defined by the symplectic splitting $S\oplus S^\perp$ is sent to the Einstein torus defined by the spacelike vector $\mu(S)\in\W$.
\begin{proof}
Let $\v{u}\in \W$ be a spacelike vector normalized so that $\v{u}\cdot \v{u}= 2$.
Then, both vectors $\v{u} \pm \omega^*$ are null.
By the fact that null vectors in $\Lambda^2(\V)$
are decomposable,
each $\v{u} \pm \omega^*$ corresponds to a $2$-plane in $\V$.
These $2$-planes are nondegenerate since
\[(\v{u}\pm\omega^*) \wedge \omega^* =
-\omega(\v{u}\pm\omega^*)\mathsf{vol} = 2\neq 0.\]
The two planes $\v{u} \pm \omega^*$ are orthogonal since they are the images of each other by the reflection $\mathsf{R}_{\omega^*}$, and so they are the summands for a symplectic splitting of $\V$.

This map is inverse to the projection $\mu$ defined above.

To prove the last statement in the proposition, we apply proposition \ref{transversePQ}. The Einstein torus defined by the splitting $S,S^\perp$ is the set of Lagrangian planes which intersect $S$ (and $S^\perp$) in a nonzero subspace. Let $P$ be such a plane. Then, $\ldot{\iota(S)}{\iota(P)}=0$, which means that
\[\ldot{\left(\iota(S) + \frac12 (\ldot{\iota(S)}{\omega^*})\omega^*\right)}{\iota(P)}=0,\]
so $\iota(P)$ is in the Einstein torus defined by the orthogonal projection $\mu(S)$. Similarly, if $\iota(P)$ is orthogonal to $\mu(S)$ then $P$ intersects $S$ in a nonzero subspace.
\end{proof}
\end{prop}
\subsection{Graphs of linear maps}
Now we describe pairs of Einstein tori in terms of symplectic splittings of
$(\V,\omega)$ more explicitly.

Let $A,B$ be vector spaces of dimension $2$ and $A\oplus B$ their
direct sum. If $A\xrightarrow{f}B$ is a linear map,
then the {\em graph\/} of $f$ is the linear subspace
$\graph(f)\subset A\oplus B$ consisting of all
$\v{a} \oplus f(\v{a})$, where $\v{a}\in A$.
Every $2$-dimensional linear subspace  $L \subset A\oplus B$ which is transverse to
$B = 0 \oplus B\subset A\oplus B$ equals $\graph(f)$ for a unique  $f$.
Furthermore, $L = \graph(f)$ is transverse to $A = A \oplus 0$ if and only if
$f$ is invertible, in which case $L = \graph(f^{-1})$ for the inverse map
$B \xrightarrow{f^{-1}} A$.

Now, suppose $A$ and $B$ are endowed with nondegenerate alternating bilinear forms $\omega_A,
\omega_B$, respectively.
Let $A \xrightarrow{f} B$ be a linear map.
Its {\em adjugate\/} is the linear map
\[ B \xrightarrow{\Adj(f)} A \]
defined as the composition
\begin{equation}\label{eq:Adjugates}
B \xrightarrow{\omega_B^\#} B^*
\xrightarrow{f^\dag} A^*
\xrightarrow{(\omega_A^\#)^{-1}} A \end{equation}
where  $\omega_A^\#, \omega_B^\#$ are isomorphisms induced by
$\omega_A, \omega_B$ respectively, and $f^\dag$ is the transpose of $f$.
If $\v{a}_1,\v{a}_2$ and $\v{b}_1,\v{b}_2$ are bases of $A$ and $B$ respectively
with
\begin{align*} \omega_A(\v{a}_1,\v{a}_2) &= 1 \\
\omega_B(\v{b}_1,\v{b}_2) &= 1,  \end{align*}
then the matrices representing $f$ and $\Adj(f)$ in these bases are related by:
\[
\Adj \bmatrix f_{11} & f_{12}  \\ f_{21} & f_{22}  \endbmatrix  \;=\;
\bmatrix f_{22} & -f_{12}  \\ -f_{21} & f_{11}  \endbmatrix.
\]
In particular, if $f$ is invertible, then
\[
\Adj(f)  =  \Det(f) f^{-1} \]
where $\Det(f)$ is defined by
$f^*(\omega_B) = \Det(f) \omega_A$.

\begin{lem}
Let $\V = S \oplus S^\perp$.
Let $S \xrightarrow{f} S^\perp$ be a linear map
and let $P = \graph(f)\subset\V$ be the corresponding
$2$-plane in $\V$ which is transverse to $S^\perp$.
\begin{itemize}
\item $P$ is nondegenerate if and only if $\Det(f) \neq -1$.
\item If $P$ is nondegenerate, then its complement $P^\perp$
is transverse to $S$, and equals the graph
\[ P^\perp = \graph\big(-\Adj(f) \big), \]
of the negative of the {\em adjugate map\/} to $f$
\[ S^\perp \xrightarrow{~-\Adj(f)}S.\]
\end{itemize}
\end{lem}
\begin{proof}
Choose a basis $\v{a}, \v{b}$ for $S$.
Then $\v{a} \oplus f(\v{a})$ and $\v{b} \oplus f(\v{b})$
define a basis for $P$, and
\begin{align*}
\omega\big(\v{a} \oplus f(\v{a}),\v{b} \oplus f(\v{b})\big) & =
\omega\big(\v{a},\v{b}\big) + \omega\big( f(\v{a}),f(\v{b})\big)
\\ & = \big(1 + \Det(f)\big) \omega\big(\v{a},\v{b}\big),
\end{align*}
since, by definition,
\[
\omega\big( f(\v{a}),f(\v{b})\big) = \Det(f)
\omega\big(\v{a},\v{b}\big). \]
Thus $P$ is nondegenerate if and only if $1 + \Det(f)\neq 0$, as desired.

For the second assertion, suppose that $P$ is nondegenerate.
Since $P,P^\perp, S, S^\perp\subset \V$ are each $2$-dimensional,
the following conditions are equivalent:
\begin{itemize}
\item $P$ is transverse to $S^\perp$;
\item $P\cap S^\perp = 0$;
\item $P^\perp + S = \V$;
\item $P^\perp$ is transverse to $S$.
\end{itemize}
Thus $P^\perp = \graph(g)$ for a linear map $S^\perp \xrightarrow{g} S$.

We express the condition that $\omega(P,P^\perp) = 0$ in terms of $f$ and $g$:
For $\v{s}\in S$ and $\v{t}\in S^\perp$, the symplectic product is zero if anly only if
\begin{equation}\label{eq:Perps}
\omega\big(\v{s} + f(\v{s}), t + g(\v{t}) \big) =
\omega\big(\v{s},  g(\v{t}) \big) + \omega\big(f(\v{s}), \v{t}\big) \end{equation}
vanishes.
This condition easily implies that $g = -\Adj(f)$ as claimed.
\end{proof}
The following proposition relates the invariant $\eta$ defined for a pair of spacelike vectors with the invariant $\Det$ associated to a pair of symplectic splittings.
\begin{prop}
  Let $S\oplus S^\perp$ be a symplectic splitting and $f:S\rightarrow S^\perp$ be a linear map with $\Det(f)\neq -1$. Let $T=\graph(f)$ be the symplectic plane defined by $f$. Then,
  \[\eta(\mu(S),\mu(T))=
  % \frac{|\ldot{\mu(S)}{\mu(T)}|}{\sqrt{(\ldot{\mu(S)}{\mu(S)})(\ldot{\mu(T)}{\mu(T)}})} =
   \frac{|1-\Det(f)|}{|1+\Det(f)|}.\]
  \begin{proof}
    Let $\v{u},\v{v}$ be a basis for $S$ such that $\omega(\v{u},\v{v})=1$. Then, $\v{u}+f(\v{u}),\v{v}+f(\v{v})$ is a basis for $T$. Moreover,
    \[\v{u}\wedge \v{v}\wedge(\v{u} + f(\v{u}))\wedge(\v{v} + f(\v{v})) = \v{u}\wedge \v{v} \wedge f(\v{u}) \wedge f(\v{v}).\]
    We can compute which multiple of $\mathsf{vol}$ this last expression represents by using the normalization $(\omega\wedge\omega)(\mathsf{vol})=-2$  and the computation
    \[(\omega\wedge\omega)(\v{u}\wedge \v{v}\wedge f(\v{u})\wedge f(\v{v})) =  2\Det(f).\]
    We deduce that
    \[\v{u}\wedge \v{v} \wedge f(\v{u}) \wedge f(\v{v})=-\Det(f)\mathsf{vol}.\]

    Using the product formula from lemma \ref{symplecticspacelike}, we find that
    \[\sqrt{2}\v{u}\wedge\v{v} + \frac{\omega^*}{\sqrt{2}}\]
    is a unit spacelike representative of $\mu(S)$, and
    \[\frac{\sqrt{2}(\v{u} + f(\v{u}))\wedge(\v{v} + f(\v{v}))}{1+\Det(f)} + \frac{\omega^*}{\sqrt{2}}\]
    is a unit spacelike representative of $\mu(T)$. Their product is
     \[\frac{-2\Det(f)}{1+\Det(f)} + 1 = \frac{1-\Det(f)}{1+\Det(f)},\]
     proving the proposition.

    % Now we compute $\ldot{\mu(S)}{\mu(T)}$ :
    % \begin{align*}
    %   \ldot{\mu(S)}{\mu(T)} &= \ldot{\left(\iota(S) + \frac12 \omega(\iota(S))\omega^*\right)}{\left(\iota(T) + \frac12 \omega(\iota(T))\omega^*\right)}\\
    %                         &= -\Det(f) + (1+\Det(f)) -\frac12(1+\Det(f))\\
    %                         &=1/2(1-\Det(f)).
    % \end{align*}
    % %
    % Finally, by the proof of lemma \ref{symplecticspacelike}, $\ldot{\mu(S)}{\mu(S)} = \frac12$ and $\ldot{\mu(T)}{\mu(T)} = \frac12(1+\Det(f))^2$. Combining these computations finishes the proof of the statement.
  \end{proof}
\end{prop}

\section{Disjoint crooked surfaces}\label{sec:crookedsurf}
In this section we apply the techniques developed above in order to prove a full disjointness criterion for pairs
of crooked surfaces.

We work in the symplectic framework of section \ref{sec:Symplectic} with the symplectic vector space $(\V,\omega)$.

Let $\v{u}_+,\v{u}_-,\v{v}_+,\v{v}_-$ be four vectors in $\V$ such that \[\omega(\v{u}_+,\v{v}_-)=\omega(\v{u}_-,\v{v}_+)=1\] and all other
products between these four vanish. This means that we have Lagrangians
\begin{align*}
  P_0 &:= \R \v{v}_+ + \R \v{v}_-,\\
  P_\infty &:= \R \v{u}_+ + \R \v{u}_-, and\\
  P_\pm &:=\R\v{v}_\pm + \R\v{u}_\pm
\end{align*}
representing the points of intersection of the photons associated to $[\v{u}_+],[\v{u}_-],[\v{v}_+],[\v{v}_-]$. We call this configuration of four points and four photons a \emph{lightlike quadrilateral}.

The \emph{crooked surface} $C$ determined by this configuration is a subset of $\Einth$ consisting of three pieces : two \emph{wings} and a \emph{stem} (see Figure \ref{fig:crookedsurface}).
The two wings are foliated by photons, and we will denote by $\mathcal{W}_+, \mathcal{W}_-$ the sets of photons covering the wings. Each wing is a subset of the light cone of $P_+$ and $P_-$, respectively. Identifying points in $\Proj(\V)$ with the photons they represent, the foliations are as follows:
\[\mathcal{W}_+ = \{[t \v{u}_+ + s \v{v}_+] ~ | ~ t s \geq 0\},\]
\[\mathcal{W}_- = \{[t \v{u}_- + s \v{v}_-] ~ | ~ t s \leq 0\}.\]
We will sometimes abuse notation and use the symbol $\mathcal{W}_\pm$ to denote the collection of points in the Einstein universe which is the union of these collections of photons.

The stem $\mathcal{S}$ is the subset of the Einstein torus determined by the splitting
$S_1 \oplus S_2 := (\R \v{u}_+ + \R \v{v}_-) \oplus (\R \v{u}_- + \R \v{v}_+)$
consisting of timelike points with respect to $P_0,P_\infty$ :
\[\mathcal{S} = \{L = \R \v{w} + \R \v{w}' ~|~ \v{w} \in S_1, \v{w'}\in S_2, |m(P_0,L,P_\infty)| = 2\}.\]
Note that this definition gives only the \emph{interior} of the stem as defined in \cite{MR3231611}. A crooked surface is the closure in $\Einth$ of a \emph{crooked plane} in the
Minkowski patch defined by the complement of the light cone of $P_\infty$
(see \cite{MR3231611}).

\begin{figure}[h]
  \centering
  \includegraphics[width=.8\textwidth]{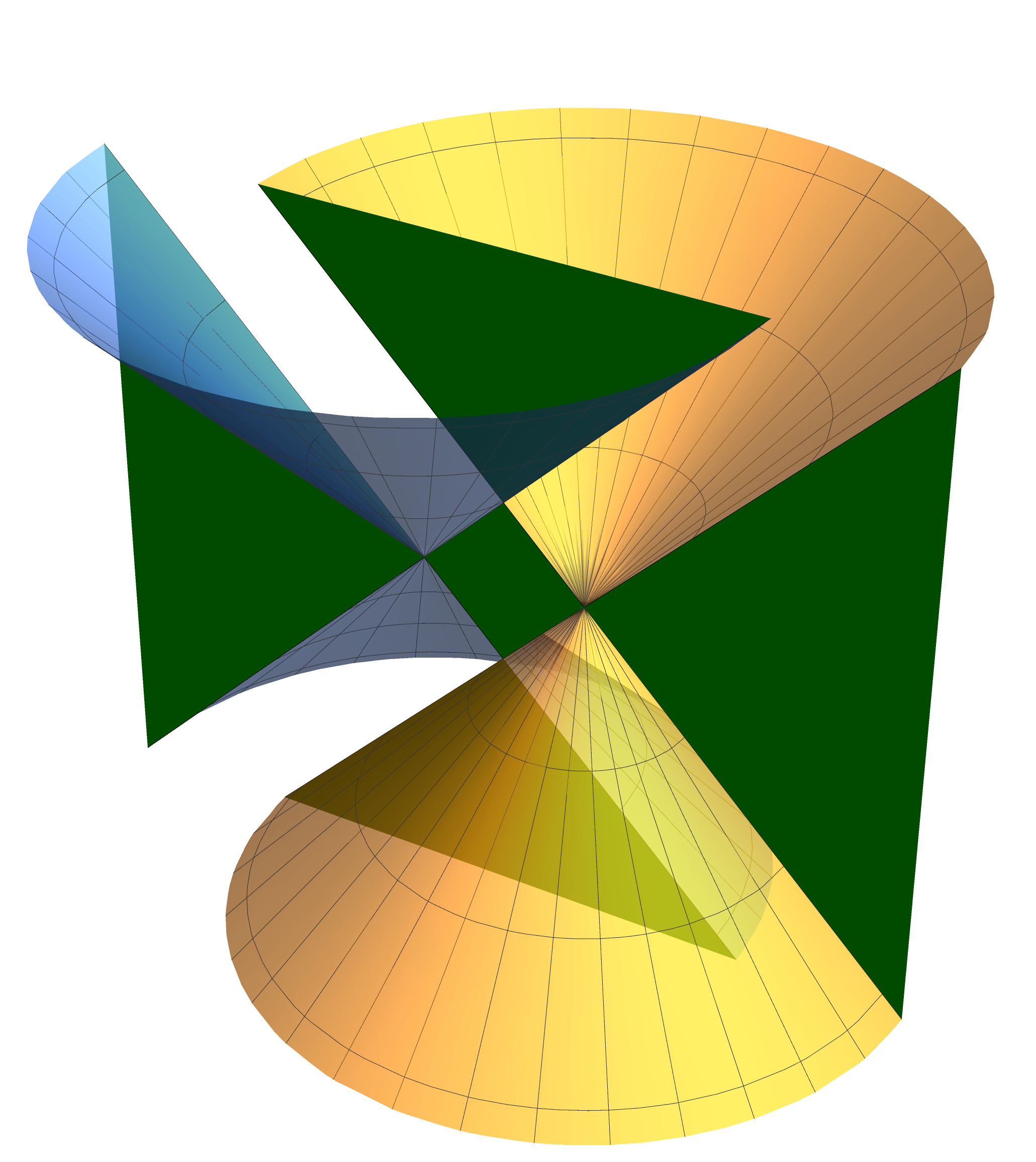}
  \caption{A crooked surface in a Minkowski patch which contains all four vertices of the lightlike quadrilateral. The Einstein torus containing the stem is a vertical plane.}
  \label{fig:crookedsurface}
\end{figure}

\begin{thm}
Let $C_1$, $C_2$ be two crooked surfaces such that their stems intersect. Then, the stem of $C_1$ intersects a wing of $C_2$ or vice versa. That is, crooked surfaces cannot intersect in their stems only.
\begin{proof}
The stem consists of two disjoint, contractible pieces. To see this, note that this set is contained in the Minkowski patch defined by $P_\infty$. There, the Einstein torus containing the stem is a timelike plane through the origin, and the timelike points in this plane form two disjoint quadrants. Let $K$ be the intersection of the two Einstein tori containing the stems of $C_1$ and $C_2$. Then, $K$ is non-contractible in either torus (Corollary \ref{Cor:IntersectionNonContractible}), so it can't be contained in the interior of the stem. Therefore, $K$ must intersect the boundary of the stem which is part of the wings.
\end{proof}
\label{nostemonly}
\end{thm}

\begin{lem}\label{lem:timelikeLightcones}
Let $p_0,p_\infty,p\in \Einth$ be three points in the Einstein universe. The point $p$ is timelike with respect to $p_0,p_\infty$ if and only if the intersection of the three light cones of $p,p_0,p_\infty$ is empty.
\begin{proof}
We work in the model of $\Einth$ given by lightlike lines in a vector space of signature $(3,2)$. If $p$ is timelike with respect to $p_0,p_\infty$, then it lies on a timelike curve which means that the subspace generated by $p,p_0,p_\infty$ has signature $(1,2)$. Therefore, its orthogonal complement is positive-definite and contains no lightlike vectors, so the intersection of the light cones is empty. The converse is similar.
\end{proof}
\end{lem}
\begin{lem}\label{lem:PhotonDisjointness}
A photon represented by a vector $\v{p}\in\V$ is disjoint from the crooked surface $C$ if and only if the following two inequalities are satisfied:
\[\omega(\v{p},\v{v}_+)\omega(\v{p},\v{u}_+)>0\]
\[\omega(\v{p},\v{v}_-)\omega(\v{p},\v{u}_-)<0.\]
\begin{proof}
Write $\v{p}$ in the basis $\v{u}_+,\v{u}_-,\v{v}_+,\v{v}_-$ :
\[\v{p} = a \v{u}_+ + b \v{u}_- + c \v{v}_+ + d \v{v}_-.\]
Then,
\[a=\omega(\v{p},\v{v}_-) \quad b=\omega(\v{p},\v{v}_+)\]
\[c=-\omega(\v{p},\v{u}_-) \quad d=-\omega(\v{p},\v{u}_+).\]
The photon $\v{p}$ is disjoint from $\mathcal{W}_+$ if and only if the following equation has no solutions with $ts\geq 0$:
\[\omega(\v{p}, t \v{u}_+ + s \v{v}_+) = 0.\]
This happens exactly when $b d < 0$.
Similarly, $\v{p}$ is disjoint from $\mathcal{W}_-$ if and only if $a c >0$. These two equations are equivalent to the ones in the statement of the Lemma, therefore it remains only to show that under these conditions, $\v{p}$ is disjoint from the stem.

The Lagrangian plane $P$ representing the intersection of $\v{p}$ with the Einstein torus containing the stem is generated by $\v{p}$ and $a \v{u}_+ + d \v{v}_-$. This is because $a \v{u}_+ + d\v{v}_-$ represents the unique photon in one of the foliations of the Einstein torus which intersects the photon $\v{p}$, and hence the span $\R\v{p} + \R(a \v{u}_+ + d \v{v}_-)$ is their intersection point, the Lagrangian $P$. We want to show that $P$ is not timelike with respect to $P_0,P_\infty$. By Lemma \ref{lem:timelikeLightcones}, this is equivalent to showing that the triple intersection of the lightcones of $P_0$, $P$, and $P_\infty$ is non-empty.

The intersection of the light cones of $P_0$ and $P_\infty$ consists of planes of the form: $\R(s \v{u}_+ + t \v{u}_-) + \R( s' \v{v}_+ + t' \v{v}_-)$ where $st' + ts'=0$. We want to show that no point represented by such a plane is incident to $P$. Two Lagrangian planes are incident when their intersection is a non-zero subspace. Equivalently, they are incident if they do not span $\V$. We have :
\begin{align*}
\det(\v{p},a \v{u}_+ + d \v{v}_-, s \v{u}_+ + t \v{u}_-, s'\v{v}_+ + t' \v{v}_-)\\ = (-bdss' + catt')\det(\v{u}_+,\v{u}_-,\v{v}_+,\v{v}_-)\\
 = k(bds^2 + act^2)\det(\v{u}_+,\v{u}_-,\v{v}_+,\v{v}_-),
\end{align*}
where $t'=kt, s'=-ks$, $k\neq 0$. There exist $t,s$ making this determinant vanish because $bd,ac$ have different signs. This means that the point where $\v{p}$ intersects the Einstein torus containing the stem is not timelike and therefore outside the stem.
\end{proof}
\end{lem}
\begin{thm}\label{CrookedSurfacesDisjointness}
Two crooked surfaces $C,C'$ given respectively by the configurations $\v{u}_+,\v{u}_-,\v{v}_+,\v{v}_-$ and $\v{u}'_+,\v{u}'_-,\v{v}'_+,\v{v}'_-$ are disjoint if and only if the four photons $\v{u}'_+,\v{u}'_-,\v{v}'_+,\v{v}'_-$ do not intersect $C$ and the four photons $\v{u}_+,\v{u}_-,\v{v}_+,\v{v}_-$ do not intersect $C'$.
\end{thm}
\begin{proof}
Since the four photons on the boundary of the stem are part of the crooked surface,  the forward implication is clear.

We now show the reverse implication. Assume that the four photons $\v{u}'_+,\v{u}'_-,\v{v}'_+,\v{v}'_-$ do not intersect $C$ and the four photons $\v{u}_+,\v{u}_-,\v{v}_+,\v{v}_-$ do not intersect $C'$. Let us first show that the wing $\mathcal{W}_+$ of $C$ does not intersect $C'$. By lemma \ref{lem:PhotonDisjointness}, it suffices to show that
\[\omega(t\v{u}_+ + s\v{v}_+, \v{v}'_+)\omega(t\v{u}_+ + s\v{v}_+, \v{u}'_+) > 0 \]
and
\[\omega(t\v{u}_+ + s\v{v}_+, \v{v}'_-)\omega(t\v{u}_+ + s\v{v}_+, \v{u}'_-) < 0 \]
for all $s,t\in \R$ such that $st\geq 0$ (with $s$ and $t$ not both zero).

We have
\begin{align*}
&\omega(t\v{u}_+ + s\v{v}_+, \v{v}'_+)\omega(t\v{u}_+ + s\v{v}_+, \v{u}'_+) \\
&=t^2\omega(\v{u}_+,\v{v}'_+)\omega(\v{u}_+,\v{u}'_+) + st\omega(\v{u}_+,\v{v}'_+)\omega(\v{v}_+,\v{u}'_+)\\ &\qquad + st\omega(\v{v}_+,\v{v}'_+)\omega(\v{u}_+,\v{u}'_+) + s^2\omega(\v{v}_+,\v{v}'_+)\omega(\v{v}_+,\v{u}'_+).
\end{align*}
By hypothesis, neither $\v{u}_+, \v{v}_+$ intersect $C'$, and neither $\v{u}'_+, \v{v}'_+$ intersect $C$. Therefore, using again lemma \ref{lem:PhotonDisjointness} and $st\geq 0$, we see that each term in this sum is non-negative and that at least one of them must be strictly positive. Therefore,
\[\omega(t\v{u}_+ + s\v{v}_+, \v{v}'_+)\omega(t\v{u}_+ + s\v{v}_+, \v{u}'_+) > 0. \]
The proof that
\[\omega(t\v{u}_+ + s\v{v}_+, \v{v}'_-)\omega(t\v{u}_+ + s\v{v}_+, \v{u}'_-) < 0 \]
is similar. Therefore, $\mathcal{W}_+$ does not intersect $C'$.

In an analogous way, one can show that $\mathcal{W}_-$ does not intersect $C'$. Therefore, the wings of the crooked surface $C$ do not intersect $C'$. Hence, to show that $C$ and $C'$ are disjoint, it only remains to show that the stem of $C$ does not intersect $C'$.

By symmetry, the wings of $C'$ do not intersect $C$, which means in particular that they do not intersect the stem of $C$. Consequently, the stem of $C$ can only intersect the stem of $C'$. However, according to theorem \ref{nostemonly}, if the stem of $C$ intersects the stem of $C'$, it must necessarily intersect its wings as well, which is not the case here. Therefore, we conclude that $C$ and $C'$ must be disjoint.
\end{proof}
By lemma \ref{lem:PhotonDisjointness}, this disjointness criterion can be expressed explicitly as $16$ inequalities (two for each of the $8$ photons defining the two crooked surfaces). There is some redundancy in these inequalities, but there does not seem to be a natural way to reduce the system.

\section{Anti-de Sitter crooked planes}
\label{sec:AdS}

In this section, we show that the criterion for disjointness of \emph{anti-de Sitter} crooked planes described in \cite{DGKfundamentald} is a special case of theorem \ref{CrookedSurfacesDisjointness}, when embedding the double cover of anti-de Sitter space in the Einstein universe.

The $3$-dimensional Anti-de Sitter space, denoted $\AdS$, is the manifold $\PSL(2,\R) \cong \Isom(\Ht)$ endowed with the bi-invariant Lorentzian metric given by the Killing form. We now recall the definition of a (right) $\AdS$ crooked plane.

Let $\ell$ be a geodesic in the hyperbolic plane $\Ht$. The \emph{right} $\AdS$ \emph{crooked plane based at the identity} associated to $\ell$ is the set of $g\in \PSL(2,\R)$ such that $g$ has a nonattracting fixed point in $\overline{\ell} \subset \Ht \cup \partial \Ht$. In other words, the isometries $g\in \PSL(2,\R)$ which make up the crooked plane are :
\begin{itemize}
  \item elliptic elements centered on a point of $\ell$,
  \item parabolic elements with fixed point in $\partial \ell$, and
  \item  hyperbolic elements with \emph{repelling} fixed point in $\partial \ell$.
\end{itemize}

A right $\AdS$ crooked plane based at $g\in \PSL(2,\R)$ is a left-translate of one based at the identity. We will say that such a crooked plane is \emph{defined by the pair} $(g,\ell)$.

A \emph{left} $\AdS$ crooked plane is defined the same way, replacing \emph{nonattracting fixed point} by \emph{nonrepelling fixed point}. Since a right $\AdS$ crooked plane and a left $\AdS$ crooked plane always intersect, we will assume in what follows that all our $\AdS$ crooked planes are of the first type.

\begin{thm}[\cite{DGKfundamentald}, Theorem 3.2]\label{thm:dgkCriterion}
Let $\ell,\ell'$ be geodesic lines of $\Ht$ and $g\in \PSL(2,\R)$. Then, the right $\AdS$ crooked planes defined by $(I,\ell)$ and $(g,\ell')$ are disjoint if and only if for any endpoints $\xi$ of $\ell$ and $\xi'$ of $\ell'$, we have $\xi\neq \xi'$ and $d(\xi,g\xi')-d(\xi,\xi')<0$.
\end{thm}
In this criterion, the difference $d(p,gq)-d(p,q)$ for $p,q\in\partial \Ht$ is defined as follows : choose sufficiently small horocycles $C,D$ through $p,q$ respectively. Then, $d(p,gq)-d(p,q):=d(C,GD)-d(C,D)$ and this quantity is independent of the choice of horocycles.

\subsection{AdS as a subspace of Ein}
Let $\V_0$ be a real two dimensional symplectic vector space with symplectic form $\omega_0$. Denote by $\V$ the four dimensional symplectic vector space $\V=\V_0 \oplus \V_0$ equipped with the symplectic form $\omega = \omega_0 \oplus -\omega_0$. This vector space $\V$ will have the same role as in section \ref{sec:Symplectic}.

The Lie group $\Sp(\V_0)=\SL(\V_0)$ is a model for the double cover of anti-de Sitter $3$-space. We will show how to embed this naturally inside the Lagrangian Grassmannian model of the Einstein Universe in three dimensions. To do this, define
\[i : \SL(\V_0) \rightarrow \mathrm{Gr}(2,\V)\]
\[f \mapsto \graph(f)\]
The graph of $f \in \Sp(\V_0)$ is a Lagrangian subspace of $\V=\V_0\oplus \V_0.$
This means that $i(\SL(\V_0))\subset \Lag(\V) \cong \Einth$. This map is equivariant with respect to the homomorphism:
\[\SL(\V_0) \times \SL(\V_0) \rightarrow \Sp(\V)\]
\[ (A,B) \mapsto B \oplus A ,\]
where the action of $\SL(\V_0) \times \SL(\V_0)$ on $\SL(\V_0)$ is by $(A,B)\cdot X = AXB^{-1}$.
The involution of $\Einth$ induced by the linear map
\[I \oplus -I : \V_0\oplus \V_0 \mapsto \V_0 \oplus \V_0,\]
where $I$ denotes the identity map on $\V_0$, preserves the image of $i$. It corresponds to the two-fold covering $\SL(\V_0)\rightarrow \PSL(\V_0)$. The fixed points of this involution are exactly the complement of the image of $i$, corresponding to the conformal boundary of AdS (see section $2$ of \cite{G14} for details).
\subsection{Crooked surfaces and AdS crooked planes}

As in \cite{G14}, we say that a crooked surface is adapted to an AdS patch if it is invariant under the involution $I \oplus -I$. Goldman proves in \cite{G14} that a crooked surface is adapted to an $\AdS$ patch if and only if it is the closure in $\Ein{3}$ of a left or right $\AdS$ crooked plane in that patch. Moreover, two $\AdS$ crooked planes in the same patch are disjoint if and only if their closures in $\Ein{3}$ are disjoint.

 If a crooked surface is invariant under $I \oplus -I$, then its corresponding lightlike quadrilateral is invariant. Two of the opposite vertices are fixed (they lie on the boundary of AdS) and the two others are swapped. If we denote the four photons by $\v{u}_-,\v{u}_+,\v{v}_-,\v{v}_+$, this means $\v{v}_- = (I \oplus -I) \v{u}_-$ and $\v{v}_+ = (I \oplus -I) \v{u}_+$.

\subsubsection{AdS crooked planes based at the identity}
For concreteness, choose a basis of $\V$ to identify it with $\R^4$.
We will represent a plane in $\R^4$ by a $4 \times 2$ matrix whose columns generate the plane, up to multiplication on the right by an invertible $2\times2$ matrix. For example, $\graph(f)$ corresponds to the matrix:
\[\begin{pmatrix}I \\ f\end{pmatrix}.\]
The identity element of $\SL(\V_0)$ maps to the plane
\[\begin{pmatrix}I\\I\end{pmatrix}\]
and its image under the involution $I \oplus -I$ is
\[\begin{pmatrix}I\\-I\end{pmatrix}.\]
The intersection of the lightcones of the two Lagrangians $\graph(I)$ and $\graph(-I)$ consists of Lagrangians which have the form
\[\begin{pmatrix} v_1 & v_1\\v_2 & v_2\\v_1 & -v_1\\ v_2 & -v_2\end{pmatrix}\]
for some $v_1,v_2\in \R$ not both zero.

Therefore, the lightlike quadrilaterals containing as opposite vertices $\graph(I)$ and $\graph(-I)$ are parameterized by pairs of distinct nonzero vectors $\v{a},\v{b} \in \V_0$ ($2\times 1$ column vectors), up to projective equivalence. The four vertices of the lightlike quadrilateral are then:
\[\begin{pmatrix} I \\ I\end{pmatrix}, \begin{pmatrix}\v{a} & \v{a} \\ \v{a}  & -\v{a} \end{pmatrix}, \begin{pmatrix}\v{b} & \v{b} \\ \v{b} & -\v{b} \end{pmatrix}, \begin{pmatrix} I \\ -I \end{pmatrix}.\]
We will say that such a lightlike quadrilateral is based at $I$ and defined by the vectors $\v{a},\v{b}$. We choose as representatives of its lightlike edges the vectors:
\[ \v{v}_- = \begin{pmatrix}\v{a} \\ \v{a} \end{pmatrix}, \v{u}_- = \begin{pmatrix} -\v{a} \\ \v{a} \end{pmatrix}\]
\[ \v{v}_+ = \begin{pmatrix}\v{b} \\ \v{b} \end{pmatrix}, \v{u}_+ = \begin{pmatrix} \v{b} \\ -\v{b} \end{pmatrix}.\]

With the definition of the wings using the sign choices of section \ref{sec:crookedsurf}, we will see that the intersection of the associated crooked surface with the $\AdS$ patch is a \emph{right} $\AdS$ crooked plane.

Indeed, the definition of the photons foliating the wing $\mathcal{W}_+$ was
\[\mathcal{W}_+ = \{[t \v{u}_+ + s \v{v}_+] ~ | ~ t s \geq 0\}.\]

Suppose that the graph Lagrangian $\begin{pmatrix} I \\ f\end{pmatrix}$ for some $f\in \SL(2,\R)$ is on such a photon, equivalently that it contains a vector of the form
\[\begin{pmatrix} (t+s)\v{b}\\(t-s)\v{b}\end{pmatrix}\]
with $ts\geq 0$. This is equivalent to
\[f\v{b} = \left(\frac{t-s}{t+s}\right)\v{b}.\]
When $t s \geq 0$, the quantity $\left|\frac{t-s}{t+s}\right| \leq 1$, hence the point $[\v{b}]\in \partial \Ht$ is a nonattracting fixed point of $f$. By a similar calculation, we can show that $\mathcal{W}_-$ consists of elements $\graph(f)$ such that $f$ has a nonattracting fixed point at $[\v{a}]\in \partial \Ht$.

\subsubsection{AdS crooked planes based at $f$}
In order to get an AdS crooked plane based at a different point $f\in \SL(\V_0)$, we map the crooked plane by an element of the isometry group $\SL(\V_0) \times \SL(\V_0) \subset \Sp(\V)$. The easiest way is to use an element of the form :
\[\begin{pmatrix}I & 0 \\ 0 & f \end{pmatrix}.\]
This corresponds to left multiplication by $f$ in $\SL(\V)$.

Applying $f$ to a lightlike quadrilateral, we get a lightlike quadrilateral with vertices of the form:
\[\begin{pmatrix} I \\ f \end{pmatrix},  \begin{pmatrix} I \\ -f \end{pmatrix} , \begin{pmatrix} \v{a} & -\v{a} \\ f\v{a} & f\v{a} \end{pmatrix},
\begin{pmatrix} \v{b} & \v{b} \\ f\v{b} & -f\v{b} \end{pmatrix} \]
and edges of the form:
\[\begin{pmatrix} \v{a} \\ f\v{a} \end{pmatrix}, \begin{pmatrix} -\v{a} \\ f\v{a} \end{pmatrix}\]
\[\begin{pmatrix} \v{b} \\ f\v{b} \end{pmatrix}, \begin{pmatrix} \v{b} \\ -f\v{b} \end{pmatrix}.\]
\subsection{Disjointness}
The disjointness criterion for crooked surfaces in the Einstein Universe is given by $16$ inequalities. Using the symmetries imposed by an AdS patch, we can reduce them to $4$ inequalities.

Using the involution defining the AdS patch, we can immediately reduce the number of inequalities by half. This is because both surfaces are preserved by the involution, and their defining photons are swapped in pairs. (So for example, we only have to check that $\v{u}_+$ and $\v{u}_-$ are disjoint from the other surface, for each surface.)

The second reduction comes from the fact that for AdS crooked planes, we only need to check that the four photons from the first crooked surface are disjoint from the second, and then the four from the second are automatically disjoint from the first.

For a crooked surface based at the identity with lightlike quadrilateral defined by the vectors $\v{a},\v{b}\in \V_0$ and another based at $f$ with quadrilateral defined by $\v{a}',\v{b}'\in \V_0$, the inequalities reduce to:
\[\omega_0(\v{a}',\v{b})^2 > \omega_0(f\v{a}',\v{b})^2\]
\[\omega_0(\v{a}',\v{a})^2 > \omega_0(f\v{a}',\v{a})^2\]
\begin{equation}\omega_0(\v{b}',\v{b})^2 > \omega_0(f\v{b}',\v{b})^2 \label{fourineqs}\end{equation}
\[\omega_0(\v{b}',\v{a})^2 > \omega_0(f\v{b}',\v{a})^2.\]
What remains is to interpret these four inequalities in terms of hyperbolic geometry. We first define an equivariant map from $\mathbb{P}(\V_0)$ to $\partial \mathbb{H}^2$. As a model of the boundary of $\mathbb{H}^2$, we use the projectivized null cone for the Killing form in $\mathfrak{sl}(\V_0)\cong \mathfrak{sl}(2,\R)$. Choose a basis of $\V_0$ in which $\omega_0$ is given by the matrix $J=\begin{pmatrix}0 & 1\\ -1 & 0\end{pmatrix}$ and define
\begin{align*}
  \eta : \V_0 &\rightarrow \N(\mathfrak{sl}(2,\R))\\
        \v{a} &\mapsto -\v{a}\v{a}^T J,
\end{align*}
where $\v{a}$ is a column vector representing a point in $\mathbb{P}(\V_0)$. This map associates to the vector $\v{a}$ the tangent vector to at identity of the photon between $I$ and the boundary point $ \begin{pmatrix}\v{a} & \v{a} \\ \v{a}  & -\v{a} \end{pmatrix}$. Note that the image of $\eta$ is contained in the upper part of the null cone.

\begin{lem}
$\eta$ is equivariant with respect to the action of $\SL(\V_0)$.
\begin{proof}
\[\eta(A\v{a}) = -A\v{a}(A\v{a})^TJ = -A\v{a}\v{a}^T A^T J = -A \v{a}\v{a}^T J A^{-1} = A \eta(\v{a}) A^{-1}.\]
\end{proof}
\end{lem}

We will denote by $K$ the \emph{trace form} on $\mathfrak{sl}(2,\R)$
\[K(X,Y) = \Tr(XY).\]
Its value is $\frac18$ times the Killing form.

\begin{lem}
Let $\v{a},\v{b} \in \V_0$. Then,
$\omega_0(\v{a},\v{b})^2 = -K(\eta(\v{a}),\eta(\v{b}))$.
\begin{proof}
\begin{align*}
\omega_0(\v{a},\v{b})^2 &= -\v{a}^T J \v{b} \v{b}^T J \v{a}\\
                         &= \v{a}^T J \eta(\v{b}) \v{a}\\
                         &= \Tr( \v{a}^T J \eta(\v{b}) \v{a})\\
                         &= \Tr(\v{a} \v{a}^T J \eta(\v{b}))\\
                         &= -\Tr(\eta(\v{a})\eta(\v{b}))\\
                         &= -K(\eta(\v{a}),\eta(\v{b})). \qedhere
\end{align*}
\end{proof}
\end{lem}
Note that the expression $\omega_0(\v{a},\v{b})$ is not projectively invariant, but the sign of $\omega_0(\v{a},\v{b})^2 - \omega_0(\v{a}, f\v{b})^2$ is.
\begin{cor}
The following inequalities are equivalent
\[\omega_0(\v{a},\v{b})^2 - \omega_0(\v{a},f\v{b})^2>0,\] \[K(\eta(\v{a}),f\eta(\v{b})f^{-1}) > K(\eta(\v{a}),\eta(\v{b})).\]
\end{cor}
Finally, we want to show that the four inequalities (\ref{fourineqs}) are equivalent to the DGK criterion (Theorem \ref{thm:dgkCriterion}).

Let $A,B,A',B'$ denote respectively $\eta(\v{a}),\eta(\v{b}),\eta(\v{a}'),\eta(\v{b}').$ Then, $A$, $B$, $A'$, $B'$ represent endpoints of two geodesics $g,g'$ in the hyperbolic plane. We want to show
\[d(\xi,f\xi'f^{-1}) - d(\xi,\xi')<0\]
for $\xi\in\{A,B\}$ and $\xi'\in\{A',B'\}$.

We use the hyperboloid model of $\Ht$, $\{X\in \mathfrak{sl}(2,\R)~|~K(X,X)=-1\}$. Consider horocycles $C_\xi(r)=\{X\in \Ht ~|~ K(X,\xi)=-r\}$ and $C_{\xi'}(r') = \{X\in \Ht ~|~ K(X,\xi')=-r'\}$ at $\xi$ and $\xi'$ respectively. The distance between these two horocycles is given by the formula
\[d(C_\xi(r),C_{\xi'}(r'))=\mathrm{arccosh}\left(-\frac12\left(\frac{K(\xi,\xi')}{2rr'} + \frac{2rr'}{K(\xi,\xi')}\right)\right).\]
Similarly,
\[d(C_\xi(r),fC_{\xi'}(r')f^{-1})=\mathrm{arccosh}\left(-\frac12\left(\frac{K(\xi,f\xi'f^{-1})}{2rr'}+\frac{2rr'}{K(\xi,f\xi'f^{-1})}\right)\right).\]
We know that $K(\xi,f\xi'f^{-1})>K(\xi,\xi')$. If $r,r'$ are sufficiently small, by increasingness of the function $x\mapsto x+\frac1x$ for $x>1$ and increasingness of $\mathrm{arccosh}$ we conclude $d(C_\xi(r),C_{\xi'}(r'))>d(C_\xi(r),fC_{\xi'}(r'))$, which is what we wanted.
\bibliographystyle{amsplain}
\bibliography{EinsteinTori}
%\bigskip

\end{document}